\documentclass[a4paper]{amsart}

\usepackage[T1]{fontenc}
\usepackage[latin1]{inputenc}
\usepackage[ngerman,english]{babel}
\usepackage{graphicx}
\usepackage{amsmath}
\usepackage{amssymb}
\usepackage{amsfonts}
\usepackage{mdwlist}
\usepackage{paralist}
\usepackage{hypbmsec}
\usepackage[colorlinks, linkcolor=black, citecolor=black, bookmarksopen=true, bookmarksopenlevel=4, final]{hyperref}  
\usepackage{aliascnt}
\usepackage{amsthm}
\usepackage{cleveref}
\usepackage[all]{xy}
\usepackage{mathtools}

\DeclareMathOperator{\Spec}{Spec}
\DeclareMathOperator{\high}{high}
\DeclareMathOperator{\medium}{medium}
\DeclareMathOperator{\homog}{homog}

\DeclareMathOperator{\Proj}{Proj }
\DeclareMathOperator{\Z}{\mathbb Z}
\DeclareMathOperator{\F}{\mathbb F}
\DeclareMathOperator{\A}{\mathbb A}

\DeclareMathOperator{\mO}{\mathcal O}
\DeclareMathOperator{\mP}{\mathbb P}
\DeclareMathOperator{\mfm}{\mathfrak m}
\DeclareMathOperator{\codim}{\mbox{codim}}

\newtheorem{thm}{Theorem}[section]

\newaliascnt{defin}{thm}

\aliascntresetthe{defin}

\newaliascnt{lemma}{thm}
\newtheorem{lemma}[lemma]{Lemma}
\aliascntresetthe{lemma}

\newaliascnt{cor}{thm}
\newtheorem{cor}[cor]{Corollary}
\aliascntresetthe{cor}

\newaliascnt{nota}{thm}

\aliascntresetthe{nota}

\newaliascnt{conject}{thm}

\aliascntresetthe{conject}

\newaliascnt{prop}{thm}

\aliascntresetthe{prop}

\newaliascnt{lemmadefin}{thm}

\aliascntresetthe{lemmadefin}

\theoremstyle{remark}
\newaliascnt{example}{thm}

\aliascntresetthe{example}

\theoremstyle{remark}
\newaliascnt{rem}{thm}
\newtheorem{rem}[rem]{Remark}
\aliascntresetthe{rem}

\crefname{equation}{Equation}{Equations}
\creflabelformat{equation}{#2(#1)#3}

\crefformat{thm}{#2Theorem~#1#3} 
\crefformat{defin}{#2Definition~#1#3} 
\crefformat{lemma}{#2Lemma~#1#3}
\crefformat{example}{#2Example~#1#3}
\crefformat{rem}{#2Remark~#1#3}
\crefformat{cor}{#2Corollary~#1#3} 
\crefformat{nota}{#2Notation~#1#3}
\crefformat{claim}{#2Claim~#1#3}
\crefformat{section}{#2Section~#1#3}
\crefformat{conject}{#2Conjecture~#1#3}
\crefformat{prop}{#2Proposition~#1#3}
\crefformat{lemmadefin}{#2Lemma/Definition~#1#3} 

\title{Bertini theorems for smooth hypersurface sections containing a subscheme over finite fields}
\author{Franziska Wutz}

\begin{document}
\pagestyle{plain}

\begin{abstract}
We show the existence of a hypersurface that contains a given closed subscheme of a projective space over a finite field and intersects a smooth quasi-projective scheme smoothly, under some condition on the dimension. This generalizes a Bertini theorem by Poonen and is the finite field analogue of a Bertini theorem by Altman and Kleiman. Furthermore, we add the possibility of modifying finitely many local conditions of the hypersurface. We show that the condition on the dimension is fulfilled for schemes with simple normal crossings and give an application to embeddings into smooth schemes.
\end{abstract}
\maketitle
\section{Introduction}

For a smooth subscheme $X$ of the projective space $\mP^n_k$ over a finite field $k$, Poonen proved the existence of smooth hypersurface sections using a geometric closed point sieve (\cite{poonen}). With this sieve method, he also proved in \cite{poonensubscheme} that the hypersurface can be assumed to contain a given closed subscheme $Z\subseteq \mP_k^n$, provided that $Z\cap X$ is smooth and $\dim X>2 \dim Z\cap X$. This was already known for infinite fields (\cite{bloch}); in this case, \cite{altmanbertini} also showed an analogue where the intersection $Z\cap X$ is not smooth, assuming $Z\subseteq X$ and another condition on the dimension. In this paper, we prove this analogue over finite fields as a special case of a result where we add the possibility to prescribe finitely many local conditions on the hypersurface. We show that the condition on the dimension is fulfilled when $Z\cap X$ is an equidimensional scheme with simple normal crossings and prove embedding results for schemes over finite fields.

Let $\F_q$ be a finite field of $q=p^a$ elements. Let $S=\mathbb F_q\left[x_0,\ldots,x_n\right]$ be the homogeneous coordinate ring of the projective space $\mP^n$ over $\F_q$ and let $S_d\subseteq S$ the $\mathbb F_q$-subspace of homogeneous polynomials of degree $d$. Let $S_{\homog}=\bigcup_{d\geq 0}S_d$ and let $S_d'$ be the set of all polynomials in $\F_q[x_0, \ldots x_n]$ of degree $\leq d$.

For a scheme $X$ of finite type over $\F_q$, we define the zeta function as
$$\zeta_X(s):=\prod\limits_{P\in X \text{ closed}} (1-q^{-s\deg P})^{-1}.$$
This product converges for $\mbox{Re}(s)>\dim X$. 

Let $Z$ be a fixed closed subscheme of $\mP^n$. For $d\in \Z_{\geq 0}$ let $I_d$ be the $\F_q$-subspace of polynomials $f\in S_d$ vanishing on $Z$, and $I_{\homog}=\bigcup_{d\geq 0}I_d$. For a polynomial $f\in I_d$ let $H_f=\Proj(S/(f))$ be the hypersurface defined by $f$.  
As in Poonen's paper (\cite{poonensubscheme}), we define the density of a subset $\mathcal P\subseteq I_{\homog}$ by
$$\mu_Z(\mathcal P):=\lim\limits_{d\rightarrow \infty}\frac{\#(\mathcal P\cap I_d)}{\#I_d},$$
if the limit exists. 
We have to use this density relative to $I_{\homog}$ and cannot measure the density using the definition of \cite{poonen}, since if the dimension of $Z$ is positive, the density of $I_{\homog}$ would always be zero (cf. Lemma 3.1 \cite{poonenirreducible}). We further define the upper and lower density $\overline {\mu}_Z(\mathcal P)$ and $\underline{\mu}_Z(\mathcal P)$ of a subset $\mathcal P\subseteq I_{\homog}$ by 
$$\overline{\mu}_Z(\mathcal P):=\limsup\limits_{d\rightarrow \infty}\frac{\#(\mathcal P\cap I_d)}{\#I_d}, \; \mbox{ and }\; \underline{\mu}_Z(\mathcal P) = \liminf\limits_{d\rightarrow \infty}\frac{\#(\mathcal P\cap I_d)}{\#I_d}.$$ 

We define the embedding dimension of $X$ at $x$ to be
$e(x)=\dim_{\kappa(x)}(\Omega^1_{X|\F_q}(x))$. Let
$$X_e=X(\Omega^1_{X|\F_q}, e)$$
be the subscheme such that a scheme morphism $f:T\rightarrow X$ factors through $X_e$ if and only if $f^*\Omega^1_{X|\F_q}$ is locally free of rank $e$. Then $X_e$ is the locally closed subscheme of $X$ where the embedding dimension of $X$ is $e$. 

Following Poonen (\cite{poonen} Lemma 1.2), we may impose local conditions on the hypersurface at a finite subscheme $Y$ of $\mP^n$. For a polynomial $f\in I_d$ we define $f\big|_Y\in H^0(Y,\mathcal I_Z \cdot \mO_Y)$ as follows: on each connected component $Y_i$ of $Y$ let $f\big|_Y$ be equal to the restriction of $x_j^{-d}f$ to $Y_i$, where $j=j(i)$ is the smallest $j\in \left\{0,1,\ldots,n\right\}$ such that the coordinate $x_j$ is invertible on $Y_i$.

\begin{thm}
\label{bertinifqtaylor}
Let $X$ be a quasi-projective subscheme of $\mP^n$ over $\F_q$ and let $Z$ be a closed subscheme of $\mP^n$. Let $Y$ be a finite subscheme of $\mP^n$, such that $U:=X-(X\cap Y)$ is smooth of dimension $m\geq 0$ and $Y\cap Z=\emptyset$. Let $V=Z\cap U$ be the intersection and let $T$ be a subset of $H^0(Y,\mathcal I_Z \cdot \mO_Y)$. If
$$\mathcal P=\left\{f\in I_{\homog}:\; H_f\cap U \mbox{ is smooth of dimension } m-1\mbox{ and } f\big|_Y\in T\right\},$$
then the following holds: 
\begin{enumerate}
\item 
If $\max\limits_{0\leq e\leq m-1} \left\{e+\dim V_e\right\}< m$ and $V_m=\emptyset$, then 
$$\mu_Z(\mathcal P)=\frac{\#T}{\#H^0(Y,\mathcal I_Z\cdot \mO_Y)}\frac{\zeta_V(m+1)}{\zeta_U(m+1)\prod\limits_{e=0}^{m-1} \zeta_{V_e}(m-e)}\neq 0.$$
\item
If $\max\limits_{0\leq e\leq m-1} \left\{e+\dim V_e\right\}\geq m$ or $V_m=\emptyset$, then $\mu_Z(\mathcal P)=0$.
\end{enumerate}
\end{thm}

The special case for $Y=\emptyset$ and $T=\{0\}$ yields the following result of \cite{wutz} (Theorem 2.1), which has meanwhile also been proved independently by Gunther (\cite{gunther} Theorem 1.1): 

\begin{thm}
\label{bertinifq}
Let $X$ be a quasi-projective subscheme of $\mP^n$ which is smooth of dimension $m\geq 0$ over $\F_q$. Let $Z$ be a closed subscheme of $\mP^n$ and let $V:=Z\cap X$ be the intersection. If
$$ \mathcal P=\left\{f\in I_{\homog}:\; H_f\cap X \mbox{ is smooth of dimension } m-1\right\},$$
then the following holds:
\begin{enumerate}
\item If $\max\limits_{0\leq e\leq m-1} \left\{e+\dim V_e\right\}< m$ and $V_m=\emptyset$, then 
$$\mu_Z(\mathcal P)=\frac{\zeta_V(m+1)}{\zeta_X(m+1)\prod\limits_{e=0}^{m-1} \zeta_{V_e}(m-e)} =\frac{1}{\zeta_{X- V}(m+1)\prod\limits_{e=0}^{m-1} \zeta_{V_e}(m-e)} \neq 0.$$
In particular, there exists a hypersurface $H$ of degree $d \gg 1$ containing $Z$ such that $H\cap X$ is smooth of dimension $m-1$.
\item
If $\max\limits_{0\leq e\leq m-1} \left\{e+\dim V_e\right\}\geq m$ or $V_m\neq\emptyset$, then $\mu_Z(\mathcal P)=0$.
\end{enumerate} 
\end{thm}

\begin{rem}
\begin{itemize}
\item [(i)]
If we choose $Z$ to be empty, then the conditions of \cref{bertinifq}(i) are fulfilled and \cref{bertinifq} gives Theorem 1.1 of \cite{poonen}.
\item [(ii)]
If the intersection $V=Z\cap X$ is smooth of dimension $l\geq 0$ as required in \cite{poonensubscheme}, then the condition on the dimension in \cref{bertinifq} implies $l+\dim V = 2 l < m$ and therefore \cref{bertinifq} (i) also yields the statement of Theorem 1.1 of \cite{poonensubscheme}.
\item [(iii)]
The density in \cref{bertinifq} is independent of the embedding $X\hookrightarrow \mP^n$.
\item[(iv)]
Note that the density is relative to $I_{\homog}$ and does not depend on points outside of $Z$; therefore we must fix $Z$ at the beginning and cannot, in general, compare two densities obtained for different closed schemes.
\end{itemize}
\end{rem}

\begin{cor}
Let $X$ be a quasi-projective subscheme of $\mP^n$ that is smooth of dimension $m\geq 0$ over $\F_q$ at all but finitely many closed points $P_1,\ldots, P_r$. Let $Z$ be a closed subscheme of $\mP^n$ that does not contain any of those points and let $V=Z\cap X$ be the intersection. Suppose $\max\limits_{0\leq e\leq m-1} \left\{e+\dim V_e\right\}< m$ and $V_m=\emptyset$. Then for $d\gg 1$, there exists a hypersurface $H$ of degree $d$ that contains $Z$ but none of the points $P_1,\ldots,P_r$, such that $H\cap X$ is smooth of dimension $m-1$.
\end{cor}
\begin{proof}
Let $Y_i=\Spec \kappa (P_i)$ and $Y=\bigcup\limits_{i=1}^r Y_i$. Then $U=X-(X\cap Y)$ is smooth of dimension $m\geq 0$, $Y\cap Z=\emptyset$ and $H^0(Y, \mathcal I_Z\cdot \mO_Y)=\prod\limits_{i=1}^r \mathcal I_{Z,P_i}\cdot\kappa(P_i)$. We define $T\subseteq H^0(Y,\mathcal I_Z\cdot \mO_Y)$ to be the nonempty set of elements that are nonzero in every component of the above product. For $f\big|_Y\in T$ this implies $P_i\notin H_f$ for all $1\leq i\leq r$, and thus $H_f\cap Y=\emptyset$. 

Applying \cref{bertinifqtaylor}, we get the existence of a hypersurface $H$ of degree $d\gg 1$ that does not intersect $Y$ and therefore contains none of the points $P_1,\ldots,P_r$; further it intersects $U$ and thus also $X$ smoothly.
\end{proof}

Let $V$ be a subscheme of $\mP^n$ and let $W_1,\ldots, W_s$ be the irreducible components of $V$. We say that $V$ has simple normal crossings if $W_i$ is smooth for any $i$, $\bigcap\limits_{i \in I}W_i$ is smooth and $\codim_V \bigcap\limits_{i\in I} W_i=\#I-1$ for any subset $I\subseteq\left\{1,\ldots, s\right\}$. 

\begin{cor}
\label{bertinisnc}
Let the notations be as in \cref{bertinifq}. Suppose $V$ is equidimensional of dimension $l$ and has simple normal crossings. If furthermore $2l <m$ holds, then there exists a hypersurface $H$ containing $Z$ such that $H\cap X$ is smooth of dimension $m-1$.
\end{cor}
\begin{proof}
We show that the conditions of \cref{bertinifq} (i) are fulfilled for the schemes $V_e=V_{l+k}$ of the flattening stratification of $V$ for $0\leq k \leq m-l$. One can prove by induction that $V_{l+k}$ is contained in the union of all intersections of $k+1$ irreducible components of $V$; more precisely, if a point $P$ is in the intersection of exactly $k$ irreducible components of $V$, then $e_V(P)\leq l+k-1$. This yields $\dim V_{l+k}+l+k\leq 2l$ for $0\leq l\leq m$. Hence, if $2l<m$ holds, then the conditions of \cref{bertinifq} are satisfied and the corollary follows. 

Note that $V_m$ is empty, since $V_{m-1}$ is contained in the union of the intersections of $m-l$ irreducible components, and this union is already of dimension zero. As $V$ has simple normal crossings, the intersection of $m-l+1$ components, which contains $V_m$, must be empty.
\end{proof}

\begin{cor}
Let $Z$ be a quasi-projective scheme over $\F_q$ satisfying
$$\max\{e+\dim Z_e\}\leq r.$$
Then $Z$ can be embedded in a smooth scheme $Y$ over $\F_q$ of dimension $r$. If $Z$ is projective, we can choose $Y$ to be projective.

In particular, if $Z$ is of dimension $l$ with simple normal crossings, then there exists a smooth $2l$-dimensional scheme $Y$ over $\F_q$ in which $Z$ can be embedded. $Y$ can be chosen projective if $Z$ is projective.
\end{cor}
\begin{proof}(cf \cite{altmanbertini} Theorem 8)
If $Z$ is projective, let $Z$ be closed in $X=\mP^n$. (For the quasi-projective case, embed $Z$ is closed in some open smooth subscheme $X\subseteq \mP^N$ of dimension $n$.) By assumption, we have $\max\{e+\dim Z_e\}\leq r$ and \cref{bertinifq} gives a hypersurface $H$ containing $Z$ which is smooth of dimension $n-1$. Inductively, we get a smooth scheme $Y$ of dimension $r$ containing $Z$. The second part follows, since $\max\{e+\dim Z_e\}\leq 2l$, as seen in the proof of \cref{bertinisnc}.
\end{proof}

The proof of \cref{bertinifqtaylor} uses the closed point sieve introduced in \cite{poonen} and is parallel to the one in \cite{poonensubscheme}. Gunther also used the sieve proof to show \cref{bertinifq}.\vspace{\baselineskip}

{\bf Acknowledgments.} The author would like to thank Uwe Jannsen and Patrick Forré for helpful discussions. Furthermore, the author gratefully acknowledges support from the Deutsche Forschungsgemeinschaft through the SFB 1085 \textit{Higher Invariants}.

\section{Singular points of low degree}

Let $\mathcal I_Z\subseteq \mO_{\mP^n}$ be the ideal sheaf of $Z$; then $I_d=H^0(\mP^n,\mathcal I_Z(d))$. As in \cite{poonensubscheme}, we fix an integer $c$ such that $S_1 I_d=I_{d+1}$ for all $d\geq c$.

\begin{lemma}
\label{surjective}
(\cite{poonensubscheme}, Lemma 2.1.)
Let $Y$ be a finite subscheme of $\mP^n$ over $\F_q$. Let
$$\phi_d:I_d=H^0(\mP^n, \mathcal I_Z(d))\rightarrow H^0(Y,\mathcal  I_Z\cdot \mO_Y(d))$$
be the map induced by the map of sheaves $\mathcal I_Z\rightarrow \mathcal I_Z\cdot \mO_Y$ on $\mP^n$. Then $\phi_d$ is surjective for $d\geq c+\dim H^0(Y,\mathcal I_Z \cdot \mathcal O_Y)$.
\end{lemma}

\begin{rem}
\label{embeddingdimension}
In the situation of \cref{bertinifqtaylor}, for a closed point $P\in V$, we have $e_{V}(P)=\dim_{\kappa(P)}\mfm_{U,P}/(\mathcal I_{Z,P},\mfm^2_{U,P});$
in particular, $\dim U\geq e_V(P)$.
\end{rem}

\begin{lemma}
\label{countingpoints}
Let $\mathfrak m\subseteq \mathcal O_U$ be the ideal sheaf of a closed point $P\in U$. Let $C\subseteq U$ be the closed subscheme of $\mP^n$ corresponding to the ideal sheaf $\mathfrak m^2\subseteq \mathcal O_U$. Then for all $d\in \Z_{\geq 0}$,
$$ \# H^0(C, \mathcal I_Z \cdot \mathcal O_C(d))=\left\{\begin{array}{cl} q^{(m+1)\deg P}, & \mbox{ if } P\notin V, \\ q^{(m-e_V(P))\deg P}, & \mbox{ if } P\in V. \end{array} \right.$$
\end{lemma}

\begin{proof}
Since $C$ is a finite scheme, we may ignore the twist, i.e. $H^0(C,\mathcal I_Z\cdot \mO_C(d))=H^0(C,\mathcal I_Z\cdot \mO_C)$. Taking cohomology of 
$0\rightarrow \mathcal I_Z\cdot \mO_C\rightarrow \mO_C\rightarrow \mO_{ Z\cap C} \rightarrow 0$
on the $0$-dimensional scheme $C$ and using \cite{hartshorne} Theorem III 2.7 yields an exact sequence
$$0\rightarrow H^0(C, \mathcal I_Z\cdot \mathcal O_{C})\rightarrow H^0(C,\mO_C)\rightarrow H^0(C,\mO_{Z\cap C})\rightarrow 0.$$
There is a filtration of $H^0(C,\mO_C)=\mO_{U,P}/\mfm_{U,P}^2$ whose quotients are vector spaces of dimensions $m$ and $1$ respectively over the residue field $\kappa(P)$ of $P$. Thus, $\#H^0(C,\mO_C)=\#\kappa(P)^{m+1}=q^{(m+1)\deg P}$. 
Next we determine $\#H^0(C,\mO_{Z\cap C})$. If $P\in U - V$, then $H^0(C,\mO_{Z\cap C})=0$. If $P\in V$, then $H^0(C,\mO_{Z\cap C})$ has a filtration whose quotients have dimensions 1 and $e_V(P)$ over $\kappa(P)$ by \cref{embeddingdimension}. Hence,
\begin{align*}
\#H^0(C,\mathcal I_Z\cdot \mO_C)
&=\left\{\begin{array}{cl} q^{(m+1)\deg P}, & \mbox{ if } P\notin V, \\ q^{(m+1)\deg P}/q^{(e_{V}(P)+1)\deg P}, & \mbox{ if } P\in V. \end{array} \right.
\end{align*}
\end{proof}

For a scheme $U$ of finite type over $\mathbb F_q$ we define $U_{<r}$ to be the set of closed points of $U$ of degree $<r$. Let $U_{>r}$ be defined similarly.

\begin{lemma}[Singularities of low degree]
\label{smalldegreetaylor}
Define
\begin{align*}
\mathcal P_r:=\{ f\in I_{\homog}: &\; H_f\cap U \mbox{ is smooth of dimension } m-1 \\
&\;\mbox{at all points } P\in U_{<r} \mbox{ and } f\big|_Y\in T\}.
\end{align*}
Then 
$$\mu_Z(\mathcal P_r)= \frac{\#T}{\#H^0(Y,\mathcal I_Z\cdot\mO_Y)}\prod\limits_{P\in (U-V)_{<r}}(1-q^{-(m+1)\deg P})\prod\limits_{e=0}^m \prod\limits_{P\in (V_e)_{<r}}(1-q^{-(m-e)\deg P}).$$
\end{lemma}

\begin{proof}
Let $U_{<r}=\left\{P_1,\ldots,P_s\right\}$. Let $\mfm_i$ be the ideal sheaf of $P_i$ on $U$ and let $C_i$ be the closed subscheme of $U$ corresponding to the ideal sheaf $\mfm^2_i\subseteq \mO_U$.  Then $H_f\cap U$ is not smooth of dimension $m-1$ at $P_i$ if and only if the restriction of $f$ to a section of $\mathcal I_Z\cdot\mO_{C_i}(d)$ is equal to zero. 
Since we also want $f\big|_Y$ to be in $T$, the set $\mathcal P_r\cap I_d$ is the inverse image of
$$T \times \prod\limits_{i=1}^s(H^0(C_i,\mathcal I_Z\cdot\mO_{C_i})\backslash\left\{0\right\})$$
under the $\F_q$-linear composition
\begin{align*}
\phi_d:I_d&=H^0(\mP^n,\mathcal I_Z(d))\rightarrow H^0(Y\cup C,\mathcal I_Z\cdot \mO_{Y\cup C}(d))\\
&\cong H^0(Y\cup C,\mathcal I_Z\cdot \mO_{Y\cup C})\cong H^0(Y,\mathcal I_Z\cdot \mO_Y)\times \prod\limits_{i=1}^s H^0(C_i,\mathcal I_Z\cdot \mO_{C_i}),
\end{align*}
where $C:=\bigcup\limits_{i=1}^s C_i$. The first isomorphism is the untwisting by multiplication by $x_j^{-d}$ component by component as in the definition of $f\big|_Z$. Note that at this point, we need the restriction $Y\cap Z=\emptyset$. 
For $d$ large enough, the map $\phi_d$ is surjective and it follows that
\begin{align*}
\mu_Z(\mathcal P_r)&=\frac{\#T}{\#H^0(Y,\mathcal I_Z\cdot \mO_Y)}\frac{\#\prod\limits_{i=1}^s(H^0(C_i,\mathcal I_Z\cdot\mO_{C_i})\backslash\left\{0\right\})}{\#\prod\limits_{i=1}^s H^0(C_i,\mathcal I_Z\cdot \mO_{C_i})}.
\end{align*}
Applying \cref{countingpoints} yields the result.
\end{proof}

\begin{cor}
Let $\max\limits_{0\leq e\leq m-1} \left\{e+\dim V_e\right\}< m$ and $V_m=\emptyset$, then
$$\lim\limits_{r\rightarrow \infty} \mu_Z(\mathcal P_r)=\frac{\#T}{\#H^0(Y,\mathcal I_Z\cdot \mO_Y)}\frac{\zeta_V(m+1)}{\zeta_U(m+1) \prod\limits_{e=0}^{m-1} \zeta_{V_e}(m-e)}.$$
\end{cor}
\begin{proof}
The first product $\prod\limits_{P\in (U-V)_{<r}}(1-q^{-(m+1)\deg P})$ in \cref{smalldegreetaylor} converges anyway, since $m+1> \dim (U-V)$. For all $0\leq e\leq m-1$, the product $\prod\limits_{P\in (V_e)_{<r}}(1-q^{-(m-e)\deg P})$ is just the partial product used in the definition of the zeta function of $V_e$. This converges for $m-e>\dim V_e$, i.e. for $\dim V_e+e<m$.
\end{proof}

\begin{proof}[Proof of \cref{bertinifqtaylor} (ii)]
The inclusion $\mathcal P\subseteq \mathcal P_r$ implies
$\mu_Z(\mathcal P) \leq \mu_Z(\mathcal P_r)$, and thus is suffices to show that $\mu_Z(\mathcal P_r)=0$. If $e+\dim V_e< m$ fails for some $e$, then the corresponding product in \cref{smalldegreetaylor}, which is the inverse of the partial product defining the zeta function of $V_e$, tends to zero as the zeta function has a pole at $\dim V_e$ (cf. \cite{tate} §4). If $V_m\neq \emptyset$, the factor $(1-q^{-(m-m)\deg P})$ appearing in the density of $\mathcal P_r$ in \cref{smalldegreetaylor} is equal to zero; hence $\mu_Z(\mathcal P_r)$ is zero. 
\end{proof}

\section{Singular points of medium degree}

\begin{lemma}
\label{lemma:mediumdegree}
Let $P\in U$ be a closed point of degree $\leq \frac{d-c}{m+1}$. Then the fraction of polynomials $f\in I_d$ such that $H_f\cap U$ is not smooth of dimension $m-1$ at $P$ is equal to
$$\left\{\begin{array}{cl} q^{-(m+1)\deg P}, & \mbox{ if } P\notin V, \\
q^{-(m-e_V(P))\deg P}, & \mbox{ if } P\in V. \end{array} \right.$$
\end{lemma}
\begin{proof}
Let $C\subseteq U$ be defined as in \cref{countingpoints}. Then $H_f\cap U$ is not smooth of dimension $m-1$ at $P$ if and only if the restriction of $f$ to a section of $\mathcal I_Z\cdot\mathcal O_C(d)$ is equal to zero. 
As we have an isomorphism $H^0(\mP^n,\mathcal I_Z(d))/\ker \phi_d \cong H^0(C,\mathcal I_Z\cdot \mO_C(d))$ by \cref{surjective}, \cref{countingpoints} applied to $C$ gives the fractions. 
\end{proof}

\begin{lemma}[Singularities of medium degree]
\label{mediumdegree}
Let
\begin{align*}
\mathcal Q_r^{\medium} :=\bigcup_{d\geq 0} \{ f\in I_d:& \; \mbox{there exists a point } P\in U \mbox{ with } r\leq \deg P\leq \frac{d-c}{m+1} \mbox{ such }\\
&\;\mbox{that } H_f\cap U \mbox{ is not smooth of dimension } m-1 \mbox{ at } P\}.
\end{align*}
Then $\lim\limits_{r \rightarrow \infty}\overline{\mu}_Z(\mathcal Q_r^{\medium})=0$.
\end{lemma}

\begin{proof}
Since the number of points $P$ of degree $g$ in $U$ is at most $\#U(\F_{q^g})$, \cref{lemma:mediumdegree} yields
\begin{align*}
\frac{\#(\mathcal Q_r^{\medium}\cap I_d)}{\#I_d}\leq \sum\limits_{e=0}^m \sum\limits_{g=r}^{\infty} \#V_e(\F_{q^g}) \;q^{-(m-e)g} + \sum\limits_{g=r}^{\infty}\#(U-V)(\F_{q^g})\;q^{-(m+1)g}.
\end{align*}
By (\cite{langweil} Lemma $1$), there exist constants $a_e$ and $a$ for $V_e$ and $U-V$ that depend only on $V_e$ and $U-V$, respectively, such that $\#V_e(\F_{q^g})\leq a_e q^{g\dim V_e}$ and $\#(U-V)(\F_{q^g})\leq a q^{g\dim(U-V)}$. Using the assumptions $\max\limits_{0\leq e\leq m-1} \left\{e+\dim V_e\right\}< m$ and $V_m=\emptyset$, we obtain 
\begin{align*}
\frac{\#(\mathcal Q_r^{\medium}\cap I_d)}{\#I_d}\leq O(q^{-r}),
\end{align*}
which tends to zero for $r\rightarrow \infty$.
\end{proof}

\section{Singular points of high degree}

\begin{lemma}[Singularities of high degree off $V$]
\label{highdegreeoffv}
Define
\begin{align*}
\mathcal Q_{U-V}^{\high}:=\bigcup\limits_{d\geq 0}\; \{\;f\in I_d: &\;\mbox{there exists a point } P\in (U-V)_{>\frac{d-c}{m+1}} \mbox{ such that  } H_f\cap U \\
& \; \mbox{is not smooth of dimension } m-1 \mbox{ at } P\}.
\end{align*}
Then $\overline{\mu}_Z(\mathcal Q^{\high}_{U-V})=0$.
\end{lemma}

\begin{proof}
This is the statement of Lemma 4.2. in \cite{poonensubscheme}; the proof does not use the fact that $V$ is smooth.
\end{proof}

\begin{lemma}[Singularities of high degree on $V$]
\label{highdegreeonv}
Define
\begin{align*}
\mathcal Q_{V}^{\high}:=\bigcup\limits_{d\geq 0}\; \{\;f\in I_d: &\;\mbox{there exists a point } P\in V_{>\frac{d-c}{m+1}} \mbox{ such that } H_f\cap U \mbox{ is not }\\
&\;\mbox{smooth of dimension } m-1 \mbox{ at } P\}.
\end{align*}
Then $\overline{\mu}_Z(\mathcal Q^{\high}_{V})=0$.
\end{lemma}

\begin{proof}
If the lemma is proven for all subsets $U_i$ of a finite open cover of $U$, then it holds for $U$ as well. Hence we can assume without loss of generality, that $U\subseteq \A^n_{\F_q}=\left\{x_0\neq 0\right\}\subseteq \mP^n_{\F_q}$ is affine. We identify $S_d$ with the space of polynomials $S_d'\subseteq \mathbb F_q[x_1,\ldots,x_n]=A$ of degree $\leq d$ by setting $x_0=1$. This dehomogenization also identifies $I_d$ with a subspace $I_d'\subseteq S_d'$.

Let $P$ be a closed point of $U$. Since $U$ is smooth, we can choose a system of local parameters $t_1,\ldots, t_n\in A$ on $\mathbb A^n$ such that $t_{m+1}=\ldots=t_n=0$ defines $U$ locally at $P$. Then $dt_1,\ldots,dt_n$ are a basis for the stalk of $\Omega^1_{\A^n|\F_q}$ at $P$ and $dt_1,\ldots,dt_m$ are a basis for the stalk of $\Omega^1_{U|\F_q}$ at $P$. We will show that the probability that $H_f \cap U$ is not smooth at a point in $V_e$ tends to zero for $d\rightarrow \infty$ and any $e$.

Consider the map $\Omega^1_{U|\F_q}\otimes \mO_{V}\rightarrow \Omega^1_{V|\F_q}$, which is surjective (\cite{hartshorne} Proposition II.8.12). Tensoring it with $\mO_{V_e}$ gives a surjective map $\phi:\Omega^1_{U|\F_q}\otimes \mO_{V_e}\rightarrow \Omega^1_{V|\F_q}\otimes \mO_{V_e}$, where $\Omega^1_{U|\F_q}\otimes \mO_{V_e}$ and $\Omega^1_{V|\F_q}\otimes \mO_{V_e}$ are locally free sheaves of rank $m$ and $e$, resp. Hence, $dt_1,\ldots,dt_{m-e}$ form a basis of the kernel of $\phi$ at $P$ and $dt_{m-e+1},\ldots,dt_{m}$ a basis of $\Omega^1_{V|\F_q,P}\otimes \mO_{V_e,P}$. In particular, $t_1,\ldots, t_{m-e}$ all vanish on V and we may assume that they vanish even on $Z$, as by definition $V=Z\cap X$.

Let $\partial_1,\ldots,\partial_n \in \mathcal T_{\A^n|\F_q,P}$ be the basis of the stalk of the tangent sheaf which is dual to $dt_1,\ldots,dt_n$. We can find an $s\in A$ with $s(P)\neq 0$ such that $D_i=s\partial_i$ gives a global derivation $A\rightarrow A$ for $i=1,\ldots,n$. 
There exists a neighbourhood $N_P$ of $P$ in $\A^n$ such that $N_P\cap U=N_P\cap \left\{t_{m+1}=\ldots=t_n=0\right\}$, $\Omega^1_{\A^n|\F_q}\big|_{N_P}=\bigoplus\limits_{i=1}^n\mO_{N_P}dt_i$ and $s\in \mO(N_P)^*$. As $U$ is quasi-compact, we can cover $U$ with finitely many $N_P$ and assume $U\subseteq N_P$. Hence in particular, $\Omega^1_{U|\F_q}=\bigoplus\limits_{i=1}^m \mO_U dt_i$.

Let $P\in V_e$ be a closed point. For a polynomial $f\in I_d$, $H_f\cap U$ is not smooth at $P$ if and only if $(D_1f)(P)=\ldots=(D_m f)(P)=0$. Note that we do not have to require $f(P)$ to be zero, since $Z$ is contained in the hypersurface $H_f$ for $f\in I'_d$, and thus $f$ vanishes at all points in $V_e\subseteq Z$ anyway.

Let $\tau =\max\limits_{1\leq i\leq {l_e}+1} (\deg t_i)$ and $\gamma=\lfloor(d-\tau)/p\rfloor$, where $l_e=\dim V_e$. We select $f_0\in I'_d$ and $g_1\in S'_{\gamma},\ldots,g_{l_e+1}\in S'_{\gamma}$ uniformly and independently at random. Then the distribution of 
$$f=f_0+g_1^pt_1 + \ldots +g_{l_e+1}^p t_{l_e+1}$$
is uniform over $I'_d$. Note that by our assumption we have $e+l_e<m$ and since $t_1,\ldots, t_{m-e}$ vanish on $Z$, we get $t_1,\ldots, t_{l_e+1} \in I'_d$.

Since the distribution of the polynomials $f$ in this representation is uniform over $I'_d$, it is enough to bound the probability for an $f$ constructed in this way to have a point $P\in V_{e,>\frac{d-c}{m+1}}$ such that $(D_1 f)(P)=\ldots=(D_{m} f)(P)=0$. Here we are using the construction above because by definition, the partial derivatives $D_if=D_i f_0+g_i^p s$ are independent of one another. We will select the polynomials $f_0,g_1,\ldots,g_{l_e+1}$ one at a time.

For $0\leq i\leq l_e+1$, define
$$ W_i=V_e\cap \left\{D_1f=\ldots =D_if=0\right\}.$$
Then $W_{l_e+1}\cap V_{e,>\frac{d-c}{m+1}}$ is the set of points $P\in V_e$ of degree $>\frac{d-c}{m+1}$ where $H_f\cap U$ may be singular. Using the induction argument of Lemma 2.6 of \cite{poonen}, one can show that if the polynomials $f_0,g_1,\ldots,g_i$ for $0\leq i\leq l_e$ have been chosen such that $\dim(W_i)\leq l_e-i$ holds, then the probability for $\dim(W_{i+1})\leq l_e-i-1$ is equal to $1-o(1)$ as $d\rightarrow \infty$, and conditioned on a choice of $f_0,g_1,\ldots, g_{l_e}$ for which $W_{l_e}$ is finite, the probability for $W_{l_e+1}\cap V_{e,>\frac{d-c}{m+1}}$ to be empty is equal to $1-o(1)$ as $d\rightarrow \infty$. Hence, the probability for a polynomial $f$ to have a point $P\in V_{e,>\frac{d-c}{m+1}}$ such that $H_f\cap U$ is not smooth at $P$ is 0. The upper density $\overline{\mu}_Z(\mathcal Q_V^{\high})$, that we actually want to calculate, is a finite sum of those probabilities, and hence zero.
\end{proof}

\begin{proof}[Proof of \cref{bertinifqtaylor}]
By definition,
$\mathcal P \subseteq \mathcal P_r \subseteq \mathcal P \cup \mathcal Q_r^{\medium} \cup \mathcal Q_{U-V}^{\high}\cup \mathcal Q_V^{\high}.$
Thus $\overline{\mu}(\mathcal P)$ and $\underline{\mu}(\mathcal P)$ differ from $\mu(\mathcal P_r)$ at most by $\overline{\mu}_Z(\mathcal Q_r^{\medium})+ \overline{\mu}_Z(\mathcal Q_{U-V}^{\high}) + \overline{\mu}_Z(\mathcal Q_V^{\high})$. Using \ref{mediumdegree}, \ref{highdegreeoffv} and \ref{highdegreeonv} for singularities of medium and high degrees, we get 
\begin{align*}
\mu_Z(\mathcal P)=\lim\limits_{r\rightarrow \infty}\mu_Z(\mathcal P_r)=\frac{\#T}{\#H^0(Y,\mO_Y)} \frac{\zeta_V(m+1)}{\zeta_U(m+1) \prod\limits_{e=0}^{m-1} \zeta_{V_e}(m-e)}.
\end{align*}
\end{proof}

\end{document}